\documentclass[12pt,reqno]{amsart}
\usepackage{amsmath,a4wide}
\usepackage{xfrac}
\usepackage{stmaryrd,mathrsfs,bm,amsthm,mathtools,yfonts,amssymb,color}
\usepackage{xcolor}
\usepackage{courier}

\begingroup
\newtheorem{theorem}{Theorem}[section]
\newtheorem{lemma}[theorem]{Lemma}
\newtheorem{proposition}[theorem]{Proposition}
\newtheorem{corollary}[theorem]{Corollary}
\endgroup

\theoremstyle{definition}
\newtheorem{definition}[theorem]{Definition}
\newtheorem{remark}[theorem]{Remark}

\numberwithin{equation}{section}
\numberwithin{subsection}{section}

\newcommand{\sing}{{\rm Sing}}

\newcommand{\cH}{{\mathcal{H}}}

\newcommand{\E}{{\mathcal{E}}}
\renewcommand{\H}{{\mathcal{H}}}
\newcommand{\Tan}{{\rm{Tan}}}

\def\XXint#1#2#3{{\setbox0=\hbox{$#1{#2#3}{\int}$ }
\vcenter{\hbox{$#2#3$ }}\kern-.6\wd0}}

\newcommand\N{{\mathbb N}}

\newcommand\R{{\mathbb R}}

\newcommand{\eps}{{\varepsilon}}

\newcommand{\bC}{{\mathbf{C}}}

\title[Quantitative estimate on singularities in isoperimetric clusters]{Quantitative estimate on singularities in isoperimetric clusters % in dimension 2 and 3
	}

\author[M.\ Colombo]{Maria Colombo}
\address{
	Institute for Theoretical Studies, ETH Z\"urich, Clausiusstrasse 47, CH-8092 Z\"urich, Switzerland
		\\
	Institut f\"ur Mathematik, Universitaet Z\"urich, Winterthurerstrasse 190, CH-8057 Z\"urich,
	Switzerland}
\email{maria.colombo@math.uzh.ch}

\author[L.Spolaor]{Luca Spolaor}
\address{Massachusetts Institute of Technology,
	Department of Mathematics,
	77 Massachusetts Avenue,
	Cambridge, MA 02139-4307, USA}
\email{lspolaor@mit.edu}


\begin{document}

\begin{abstract}
We prove a quantitative estimate on the number of certain singularities in almost minimizing clusters. In particular, we consider the singular points belonging to the lowest stratum of the Federer-Almgren stratification (namely, where each tangent cone does not split a $\R$) with maximal density.
As a consequence we obtain an estimate on the number of triple junctions in $2$-dimensional clusters and on the number of tetrahedral points in $3$ dimensions,
that in turn implies that the boundaries of volume-constrained minimizing clusters
form at most a finite number of equivalence classes modulo homeomorphism of the boundary, provided that the prescribed volumes vary in a compact set.

The method is quite general and applies also to other problems: for instance, to count the number of singularities in a codimension 1 area-minimizing surface in $\R^8$.%, and to the Steiner tree problem.

%\\ \\
%{\bf Mathematics Subject Classification:} TO BE WRITTEN

\end{abstract}

\maketitle

\section{Introduction}
In this paper we consider one of the most famous examples of stratified singularities, namely the boundaries of almost minimizing bubble clusters in $\R^n$. They were first studied by Almgren, who proved the existence and regularity of isoperimetric bubble clusters up to a set of dimension $(n-2)$ (see Theorem~\ref{thm:noto-cluster} below), and to Taylor who gave a complete description of minimizing clusters in dimension $3$.
In this context, she exploited the idea of stratification of the singular set, proving that the singular set is made of a finite number of smooth surfaces, meeting along $C^{1, \alpha}$ curves with angles of $120$ degrees, that in turn meet at some points and form isolated tetrahedral singularities.

The idea of splitting the singular set according to the number of symmetries of the tangent plane was first introduced in the context of dimension-reduction by Federer, and then developed by Almgren in its fundamental contribution~\cite{almgrenBIG}, in order to study the regularity of $Q$-valued harmonic maps. For a more recent presentation, we refer to the work of White \cite{white97}.
In our context, given an almost minimizing cluster $\E$ in $\R^n$, so that $\partial \E$ has locally finite $(n-1)$-Hausdorff measure, this corresponds to consider for every $k=0,...,n-1$ the sets
$$\Sigma^{k}(\partial \E) = \{ x: \textit{every tangent cone at $x$ has at most $k$ symmetries}\}$$
(see Section~\ref{sec:prelim} for more precise definitions). They are an increasing family of sets with respect to $k$ with the property that
$${\rm dim} (\Sigma^{k}(\partial \E)) \leq k.$$	
Recently, the idea of quantitative stratification has been further analyzed by Naber and Valtorta \cite{nava1,nava2} both in the context of harmonic maps and in the one of stationary varifolds. They were able to prove that each stratum is a rectifiable set and that, under the minimality assumption, the biggest stratum of singular points has not only dimension $(n-3)$ or $(n-7)$ respectively, but also finite Hausdorff measure.  

In this paper, given an almost minimizing cluster $\E$ we make a quantitative estimate on the number of singular points with given density $\Theta_0 \in (1,\infty)$ in the $0$-stratum, namely on the set
\begin{equation}
\label{eqn:set}{\Sigma}^{0, \Theta_0}({\partial \E}) = \big \{ x \in \Sigma^{0}({\partial \E}):  \Theta_{\partial \E}(x) = \Theta_0 \big\},
\end{equation}
where we recall that the density of the boundary $\partial \E$ at a point $x \in \partial \E$ exists by the monotonicity formula (see Theorem~\ref{thm:monot} below) and is given by
$$\Theta_{\partial \E} (x) = \lim_{r\to 0^+} \frac{P(\E; B_r)}{\omega_{n-1} r^{n-1}} = \lim_{r\to 0^+} \frac{\H^{n-1} (\partial \E\cap B_r)}{\omega_{n-1}r^{n-1}}$$
(where $\omega_{n-1}$ denotes the volume of the unit ball in $\R^{n-1}$). 
The constant $\Theta_0$ is chosen to be the maximal density of area-minimizing cones in $\R^n$
\begin{equation}
	\label{def:theta_0}
	\Theta_0 = \max \big\{ \Theta_\bC(0): \mbox{ $\bC \subseteq \R^n$ is a cone-like minimizing cluster} \big\}
\end{equation}
and it is assumed to be strictly greater than the density of any cone with at least one simmetry
\begin{equation}
\label{hp:theta_0}
\begin{split}
\Theta_0& > \max \big\{ \Theta_\bC(0): \mbox{ $\bC \subseteq \R^n$ is a cone-like minimizing cluster with at least a symmetry} \big\}\\
&= \max \big\{ \Theta_\bC(0): \mbox{ $\bC \subseteq \R^{n-1}$ is a cone-like minimizing cluster} \big\}
\end{split}
\end{equation}
(a more precise definition of symmetry can be found in \eqref{e:subspace}). We notice that this assumption is satisfied, for instance, in the case of isoperimetric clusters in dimension $2$ and $3$, where the cone-like minimizing clusters are classified, the density is constant in each stratum and decreasing with respect to the stratum.

One can easily see by a contradiction argument via blowup that this set is discrete; Proposition~\ref{prop:quant-discreteness} quantifies this fact by showing that if $x\in {\Sigma}^{0, \Theta_0}({\partial \E})$ and 
$$B (x, r/2) \setminus  B (x, \lambda r) \cap {\Sigma}^{0, \Theta_0}({\partial \E}) \neq \emptyset$$
for any $r$ sufficiently small and for a suitable $\lambda$, then the quantity appearing in the monotonicity formula drops of a fixed amount between $r$ and $\lambda^2 r$.
This, together with a covering argument on the singular set first introduced by Cheeger and Naber \cite{chna} and then revisited by Naber and Valtorta \cite{nava3}, allows to prove the following result.
\begin{theorem}\label{thm:main}
	Let $n,N\in \N$, $0< r_0, \Lambda$, $\Theta_0 \in (1,\infty)$ be defined in \eqref{def:theta_0} and satisfying \eqref{hp:theta_0}, $\E$ be a $(\Lambda, r_0)$-minimizing cluster and let $\Sigma^{0, \Theta_0}(\partial \E)$ be the set defined in \eqref{eqn:set}. Then there exists $C:= C(n, \Lambda r_0)$,  $c:= c(n, \Lambda r_0)$ such that for every $R\in (0, cr_0)$ and $x_0 \in \R^n$
	$$\H^{0} \big(\Sigma^{0, \Theta_0}(\partial \E)\cap B_{R}(x_0)\big) \leq C^{\frac{P(\E; B_{2R}(x_0))}{R^{n-1}}}.$$
\end{theorem}
%	Although the statement of Theorem~\ref{thm:main} is about almost minimizers, it can  be easily proved also for stationary varifolds.
%More in general, the idea behind this theorem applies to several different contexts; for instance a similar result gives an estimate on the number of singularities in the lowest stratum with given density in a codimension $1$ minimal surface in $\R^n$. 

%\subsection{Application to the volume-constained minimizing clusters in dimension 2 and 3}

In dimension $2$ and $3$, the possible blow-ups of the boundary are classified, and correspond up to isometries to triple junctions (in dimension $2$ and $3$) and to tetrahedral singularities (in dimension $3$). If we set $\theta_0$ to be the density of triple junctions and tetraedral points, in dimension $2$ and $3$ respectively, the density of the cones in the minimal stratum is always the same and $\Sigma^{0}(\partial \E)= \Sigma^{0, \Theta_0}({\partial \E})$. Theorem ~\ref{thm:main} implies then the following:
\begin{corollary}\label{cor:cluster23}
	Let $n=2$ or $n=3$, $N\in \N$, $m_0 \leq M_0$, $m \in [m_0, M_0]^n$.
	Then there exists a constant $C:= C(N, m_0, M_0)$ such that each solution $\E$ of the isoperimetric problem
		\begin{equation}
	\label{isoperimetric problem}
	\inf \big\{ P(\E): \E \mbox{ is an $N$-cluster in $\R^n$ with }m (\E) = m \big\}\,,
	\end{equation}
 satisfies the estimate
	\begin{equation}
	\label{ts:quant-cluster}
	\H^{0} \big(\Sigma^{0}({\partial \E})\big) \leq C.
	\end{equation}
Therefore the set of isoperimetric clusters as $m$ varies in $[m_0, M_0]^n$
can be split in a finite number of equivalence classes according to homeomorphisms of their boundary.	
\end{corollary}

The previous result was obtained, with considerably more effort, in \cite{CiLeMaIC1,CiLeMaIC2} as a consequence of an improved convergence theorem for minimizing clusters, and it is inspired by a list of questions concerning partitioning problems proposed by Almgren in \cite[VI.1(6)]{Almgren76}, precisely ``to classify in some reasonable way the different minimizing clusters corresponding to different choices of $m\in \R_+^N$''. It can be shown that, inside each equivalence class, one can actually build $C^{1,1}$-diffeomorphisms, but this would require at least in dimension $3$ more technique, as developed in \cite{CiLeMaIC1,CiLeMaIC2}, and goes beyond the purpose of this paper.

A consequence of the previous corollary is that, in dimension $n=2$ and $3$, the solutions of the isoperimetric problem with volumes $m$ have a bounded number of connected components of $\R^n \setminus \partial \E$ as $m$ varies in $[m_0, M_0]^n$. For instance, when $N=2$, the number of edges is exactly $3 \H^{0} \big(\Sigma^{0}({\partial \E}))/2$, and the number of faces can be estimated by noticing that each edge is in common to at most $2$ faces. So the number of  connected components of $\R^n \setminus \partial \E$ is estimated by  $3 \H^{0} \big(\Sigma^{0}({\partial \E}))$. A similar argument works also in dimension $3$.

\bigskip

The method used to show Theorem~\ref{thm:main} is quite general and applies to other problems with stratified singularities. For instance, in a totally analogous way we can count singularities in the $0$-stratum with maximal density for area-minimizing hypersurfaces in any dimension. A simpler result can be obtained when the singular set is discrete; indeed, in this case we don't need any assumption on the density of the tangent cones to obtain an estimate on the number of singularities. For instance, a corollary of our method is the following theorem, which refines the volume bound in \cite{nava2} in the $8$-dimensional case:
\begin{theorem}
\label{thm:8}
Let $E$ be a set locally minimizing the perimeter functional in $\R^8$.  Then there exists a universal constant $C$ such that for every $R>0$ and $x_0 \in \R^n$
$$\H^{0} \big({\rm Sing}(\partial E)\cap B_R(x_0)\big) \leq C^{\frac{P(E; B_{2R}(x_0))}{R^{n-1}}}.$$

\end{theorem}

%{\color{red} Per dire che sono un numero finito di classi di equivalenza a meno di diffeo $C^{1,1}$, bisogna dire che dati due grafi con lo stesso numero di vertici e lo stesso numero di connessioni, che si attaccano a 120 gradi, esiste un diffeo $C^{1,1}$ che manda l'uno nell'altro.}

%\subsection{Application to the Steiner tree problem} 

\noindent{\bf Acknowledgments.} The authors thank Guido De Philippis for helpful conversations. M. Colombo acknowledges the support of Dr. Max R\"ossler, of the Walter Haefner Foundation and of the ETH Zurich Foundation.  The research of L. Spolaor has been supported by the Max Planck Institute for Mathematics in the Sciences (MIS) in Leipzig.

\section{Notation and preliminaries}\label{sec:prelim}
A $N$-cluster, or simply a cluster, is a family $\E=\{\E(h)\}_{h=1}^N$ of disjoint sets of $\R^n$, called the chambers of $\E$. A cone-like cluster is a cluster $\E=\{\E(h)\}_{h=1}^N$ such that $\E(h) = r \E(h)$ for every $r>0$ and $h=1,...,N$. %The set  $\E(0)=\R^n\setminus\bigcup_{h=1}^N\E(h)$ is called the exterior chamber of $\E$. 
The volume of $\E$ is the vector $m(\E)=(|\E(1)|,...,|\E(N)|)$. 
%The {\it relative distance between the $N$-clusters $\E$ and $\E'$ in $\Om\subseteq \R^n$} is defined by
%$$
%d_\Om(\E, \E') = \sum_{h=0}^N |\Om \cap (\E(h) \Delta \E'(h))|\,.
%$$
The {\it relative perimeter $P(\E;\Omega)$ of the cluster $\E$ in any open set $\Omega$} is defined as
$$
P(\E; \Omega) = \frac{1}{2} \sum_{i=1}^M P(\E(i);\Omega)\,,
$$
so that $P(\E)=P(\E;\R^n)$. 
%We say that a sequence $\{\E_k\}_{k\in\N}$ of $N$-clusters converges in $L^1_{{\rm loc}}(\Om)$ to a $N$-cluster $\E$ if $1_{\E_k(h)}\to 1_{\E(h)}$ in $L^1_{\rm loc}(\Om)$ for every $h=1,...,N$. If $\sup_{k\in\N}P_s(\E_k;\Om)<\infty$, then one can find a subsequence of $\{\E_k\}_{k\in\N}$ which admits an $L^1_{{\rm loc}}(\Om)$ limit. 
Finally, the boundary of a Borel set $E\subset\R^n$ is defined as
\begin{equation}
\label{boundary of E set}
\partial E=\Big\{x\in\R^n:0<|E\cap B_r(x)|<|B_r(x)|\mbox{ for every } r>0\Big\}\,.
\end{equation}
and the boundary of a cluster as %In this way \eqref{boundary of E cluster} is equivalent to
\[
\partial\E=\bigcup_{h=1}^N\partial\E(h)\,.
\]
We define the set of regular points ${\rm Reg} \,  \E $ as the set of points $x\in \partial \E$ such that there exists a neighborhood of $x$ where $\partial \E$ is an embededded $C^{1,1}$-hypersurface.

A natural object to study the regularity of clusters are the so-called almost minimizers. Given  $\Lambda\in [0,\infty)$ and $r_0 \in (0,\infty]$, a $N$-cluster $\E$ in $\R^n$ is a 
$(\Lambda,r_0)$-minimizing cluster if for every $x \in \R^n$, $r \leq r_0$, the inequality
$$P(\mathcal F) \leq P(\E) + \Lambda r^n$$
holds for every cluster $\mathcal F$ such that $\mathcal F(i) \cap B_r^c= \E(i)\cap B_r^c$ for $i=1,..., N$. 

A fundamental tool to prove the regularity of the boundary of almost minimizing clusters is the monotonicity formula.
\begin{theorem}[Monotonicity formula for clusters]\label{thm:monot} Let $\Lambda\in [0,\infty)$ and $r_0 \in (0,\infty]$.
If $\E$ is a $(\Lambda,r_0)$-minimizing cluster, the quantity
$$M_\Lambda(\E, x, r) = \frac{e^{\Lambda r}P(\E; B_r)}{r^{n-1}}$$
is nondecreasing for $r\leq r_0$ and for every $x\in \partial \E$. Moreover, if $\Lambda =0$ and for two radii $r_1 < r_2$ the quantity $M_0$ is constant between them, namely $M_0(\E, x, r_1) = M_0(\E, x, r_2)$, then $\E$ coincides with a cone in the annulus $ B_{r_2}(x) \setminus B_{r_1}(x)$.
\end{theorem}
From the monotonicity formula it follows that at each point $x \in \partial \E$ of a $(\Lambda,r_0)$-minimizing cluster, the blow-ups 
$( \E - x)/{r}$
converge in $L^1$ up to subsequence to a $(0,\infty)$-minimizing cone-like cluster $\bC$. Moreover, the boundaries $(\partial \E - x)/{r}$ converge to $\partial \bC$ in the Hausdorff sense.
This motivates the definition of the set of tangent cones at the point $x$, denoted by $\Tan(x, \E)$, as the the set of possible limits of  $(\E - x)/{r}$.
For the same reason, we consider the cone-like $(0,\infty)$-minimizing clusters of $\R^n$, that appear also in the definition of $\Theta_0$ in \eqref{def:theta_0}.

Next, we split the singular set of $\partial \E$ according to the maximal number of symmetries of its tangent cones. To this end, we recall that $\Theta_\bC(0)\geq \Theta_\bC(x)$ for every $x\in \R^n$ by upper semicontinuity of the density and we define $L_{\partial\bC}$ as the set where equality is realized
\begin{equation}%\label{def:LC}
L_{\partial\bC}:=\{x\in \partial\bC\,:\, \Theta_{\partial\bC}(x)=\Theta_{\partial\bC}(0)\}\,. \label{e:subspace}
\end{equation}
The dimension of $L_{\partial\bC}$ describes the number of symmetries of the cone $\partial\bC$. The set $L_{\partial\bC}$ enjoys the following properties. 
\begin{lemma}\label{l:cones} Let $\bC$ be a cone-like cluster in $\R^n$. Then $L_{\partial\bC}$ defined in \eqref{e:subspace} is a linear subspace of $\R^{n+1}$ and, if we denote by $k \in \{0,...,n-1\}$ the dimension of $L_{\partial\bC}$, then there exists a  $(0,\infty)$-minimizing cone $\bC'$ in $\R^{n-k}$ such that $\bC=\bC' \times L_{\partial\bC}$.
\end{lemma}
The previous lemma allows to split the singular set according to the dimension of the vector spaces $L_{\partial\bC}$, where $\bC$ is any tangent cone at $x$.
%
%Tangent cones exists at every point. Indeed, notice that by (i) of Lemma \ref{l:USC}, 
%\[
%\cH^{n}(M_{x,\rho}\cap B_1)=\frac{\cH^n(M\cap B_\rho(x))}{\rho^n}\leq \cH^n(M\cap B_1)<\infty\,,
%\]
%so that we can apply (iii) of Proposition \ref{p:class} to any sequence $\{M_{x,\rho_j}\}_j$ ($\rho_j\to 0$ as $j\to \infty$) and find a subsequence $\{M_{x,\rho_{j'}}\}_{j'}$ and an element $\bC\in \cM$, such that $M_{x, \rho_{j'}}\to \bC$. Furthermore taking the limit in the previous identity we conclude that $\Theta_\bC(x, rho)=\Theta_M(x)=\Theta_\bC(0)$ for every $\rho>0$, thus showing that $\bC$ is a cone by (iii) of Lemma \ref{l:USC}.

\begin{definition}[Stratification]
Given an almost minimizing cluster $\E$, we define for every $k\leq n$ the \textit{$k$-th stratum of $\partial \E$} by
\[
\Sigma^k(\partial \E):=\{x\in \sing(\partial \E)\,:\, \dim(L_{\partial \bC})\leq k\,\text{for every } \bC\in \Tan(x, \E)\}\,.
\] 
\end{definition}

Trivially the inclusion $\Sigma^k(\partial \E) \subset \Sigma^{k+1}(\partial \E)$ holds and it is known that for every $k=1,..., n$
\[
\dim_\cH(\Sigma^k(\partial \E))\leq k \quad \forall k=0,\dots,n\,,
\]
namely $\cH^{k+\alpha}(\Sigma^k(\partial \E))=0$ for every $\alpha>0$.% In particular $\Sigma^0(\partial \E)\cap \{\Theta_M(x)=\alpha\}$ is locally finite for every $\alpha>0$.

The highest stratum coincides with the set of regular points ${\rm Reg} \, \E$  and a posteriori it corresponds to the union of the reduced boundaries, in the sense of De Giorgi, of the chambers of $\E$. More precisely it consists in the set of boundary points where the blow-up is a couple of complimentary half-spaces, which in turn is a locally $C^{1,1}$ set in the context of almost minimizing clusters.

The following Theorem, due to Almgren \cite{Almgren68,Almgren76}, shows the existence and regularity of an area-minimizing cluster with volume constraint.
\begin{theorem}[Existence of isoperimetric clusters, regularity and almost minimality]\label{thm:noto-cluster}
	For every $m\in\R^N_+$ there exists a bounded isoperimetric cluster $\E$ of $\R^n$, namely a minimizer of \eqref{isoperimetric problem},
	and the diameter of $\E$ is uniformly bounded as soon as $m$ varies in a compact subset of  $\R^N_+$.
	%If we set the boundary of $\E $ to be
	%	and $\partial^* \E$ the set of boundary points where the blow-up is a couple of complimentary half-spaces, then 
	Moreover, ${\rm Reg} \E$ is a finite union of analytic hypersurfaces with constant mean curvature in $\R^n$, which is relatively open in $\partial\E$ and the complement ${\rm Sing }\, \E = \partial\E\setminus {\rm Reg} \E$ satisfies the bound on the Hausdorff dimension
	$${\rm dim} ({\rm Sing }\, \E) \leq n-2.$$	
	Moreover, let $0<m_0<M_0$;  there exist $\Lambda,r_0>0$ depending only on $m_0, M_0$ such that each minimizer $\E$ of \eqref{isoperimetric problem} is a $(\Lambda,r_0)$-minimizer if $m \in [m_0, M_0]^N$.
\end{theorem}

The less known part of the previous statement is perhaps the last sentence on almost-minimality, which follows by a contradiction argument employing the so called ``volume-fixing variations (for the sake of completeness, we mention that this short argument is presented in \cite[Proof of Theorem 1.10]{CiLeMaIC1}).

In dimension $n=2$, the only area-minimizing cone, up to isometries, is the triple junction, namely a set made by three disjoint sectors of 120 degrees, meeting at the origin.
As a consequence, the following classical result on the structure of any almost-minimizing cluster holds (see \cite[Theorem 30.7]{maggibook} or \cite[Theorem 5.2]{CiLeMaIC1}).
\begin{theorem}[Structure of isoperimetric clusters in $\R^2$]\label{thm:noto2}
If $\E$ is a $(\Lambda,r_0)$-minimizing cluster for some $\Lambda \in [0,\infty)$, $r_0 \in (0,\infty]$%in problem \eqref{isoperimetric problem} 
in dimension $n=2$, ${\rm Reg}\, \E$ is a locally finite union of $C^{1,1}$-curves (with the diameter of each curve estimated from below by $1/2\Lambda$ as soon as the curve has empty boundary), the set $\Sigma^0({\partial \E})$ coincides with the whole singular set and is discrete. Finally, for each $x \in \Sigma^0({\partial \E})$ there exist exactly three $C^{1,1}$ curves, belonging to three different interfaces, which share x as one of their endpoints.
\end{theorem}
From this theorem it follows easily that every minimizing cluster in problem \eqref{isoperimetric problem} in dimension $n=2$  is a finite union of open circular arcs (with nonempty boundary), meeting at triple junctions with angles of $120$ degrees.

In dimension $n=3$, it has been shown by Taylor \cite{taylor76} (see also \cite[Theorem 1.1]{CiLeMaIC2}) that the only cones locally minimizing the area-functional are, up to isometries: a reference closed cone $Y$ in $\R^3$ defined by three half-planes meeting along their common boundary line (which contains the origin of $\R^3$) by forming 120 degrees angles, and a reference closed cone $T$ in $\R^3$ spanned by edges of a regular tetrahedron and with vertex at the barycenter of the tetrahedron (which is assumed to be the origin of $\R^3$). 
Correspondingly,  at every singular point in ${\rm Sing} \, \E$ of an almost-minimizing cluster $\E$, we must have a unique blow-up, either of the type $Y$ or of the type $T$, and the structure of the singular set can be understood thanks to a epiperimetric inequality at triple junctions.
\begin{theorem}[Taylor's description of isoperimetric clusters in $\R^3$]\label{thm:noto3} Let $\Lambda \in [0,\infty), r_0 \in (0,\infty]$. There exists $\alpha \in (0,1)$ with the following property. If $\E$ is a $(\Lambda, r_0)$-minimizing cluster in $\R^3$ then the blow-up at each boundary point is unique and
$$
\Sigma^0(\partial \E)= \{ x \in \partial \E : \textit{ the blow-up of $\partial \E$ at $x$ is isometric to $T$} \},
$$
$$
\Sigma^1(\partial \E) \setminus \Sigma^0(\partial \E)= \{ x \in \partial \E : \textit{ the blow-up of $\partial \E$ at $x$ is isometric to $Y$} \}.
$$
Moreover $\Sigma^0(\partial \E)$ is locally finite, there exists a locally finite family $ S(\E) $ of closed connected topological surfaces with boundary in $\R^3$ and a locally finite family $\Gamma(\E)$ of closed connected $C^{1,\alpha}$-curves with boundary
such that
$S^*= S \setminus \Sigma^0(\E)$
is a $C^{1,\alpha}$-surface with boundary in $\R^3$ for every $S \in S(\E)$, 
$$ \partial \E = \bigcup_{S \in S(\E)} S, \qquad 
 {\rm Reg}\, \partial \E = \bigcup_{S \in S(\E)} {\rm int} (S^*), 
\qquad
\Sigma^0 (\partial \E) = \bigcup_{S \in S(\E)} {\rm bd} (S^*),
$$
$$\Sigma^1(\partial \E) = \bigcup_{\gamma \in \Gamma(\E)} {\rm int} (\gamma),
\qquad
\Sigma^0(\partial \E) = \bigcup_{\gamma \in \Gamma(\E)} {\rm bd} (\gamma).
$$
%Finally, $\bd \gamma \neq \emptyset for every $\gamma$ in the previous decomposition if $\Lambda = 0$, $r_0 =\infty$.
\end{theorem}

A similar characterization of the singular set in dimension $4$ or more is currently an open problem, especially because we don't have any description of the singular cones.
For instance, an interesting problem would be to characterize the minimizing cones in $\R^4$ without symmetries; this would in turn describe the possible blow-ups of almost minimizing clusters. %A contradiction argument similar to the one in Proposition ... allows to show that 

\section{Proofs}
Before proving Theorem~\ref{thm:main}, we show a covering lemma, that was previously employed in similar formulations in \cite{nava3,ghsp} and a proposition, showing that whenever there is a singular point in $\Sigma^{0,\Theta_0}(\partial \E)$ in a certain annulus around a given singular point in the same set, then the monotonicity formula drops of a fixed amount.
\begin{lemma}[Covering lemma]\label{lemma:covering}
	Let $X \subseteq B_{1/2} \subseteq \R^n$ be a collection of points and let $N\in \N$, $\lambda \in (0,1/5)$. Assume that for every $x\in X$ there exists a collection of scales $S_x \subseteq \N \cup\{0\}$ such that $|S_x| \leq N$ and
	\begin{equation}
	\label{hp:annulus}
	X \setminus \{x\} \subseteq \bigcup_{j\in S_x} B_{\lambda^j} \setminus B_{\lambda^{j+1}}.
	\end{equation}
	Then 
	\begin{equation}
	\label{ts:cover}
	\H^0(X) \leq \Big( \frac{ 10}{\lambda^2}\Big)^{nN}.
	\end{equation}
	%there exists $C:= C(n, \lambda)$ such that
	%$$\H^0(X) \leq C^N.$$, 
\end{lemma}

\begin{proof}
	The thesis is equivalent to prove that, given a set $X \subseteq B_{1/2}$ with more than $(10 \lambda^{-2})^{nN}$ elements, there exists an element $x\in X$ such that the number of annulai occupied by points in $X$ is strictly greater than $N$, namely
	$$ \H^0 \big(\big\{i\in \N: X \cap \big(B_{\lambda^{i}}(x) \setminus B_{\lambda^{i+1}}(x) \big) \neq \emptyset \big\} \big)>N.$$
	We prove this statement by induction on $N$. If $N =0$ the statement holds. Let us assume that the statement holds for $N-1$ and let us prove it for $N$. 
	Let $d$ be the diameter of $X$, namely $d:= \sup_{x,y \in X} |x-y|$; it follows that
\begin{equation}
\label{simo}
X \cap \big( B_d(x) \setminus B_{d/2} (x) \big) \neq \emptyset \qquad \mbox{for every }x\in X.
\end{equation}
	If $d \in [\lambda^{k+1}, \lambda^{k})$ for some $k\in \N$ we deduce that $ d/2 \in [\lambda^{k+2}, \lambda^{k})$. From \eqref{simo}, for every $x\in X$ at least one of the two annulai $B_{\lambda^{k}}(x)\setminus B_{\lambda^{k+1}}(x)$ and $B_{\lambda^{k+1}}(x)\setminus B_{\lambda^{k+2}}(x)$ contains an element of $X$:
	\begin{equation}
	\label{eqn:alternativa}
	\big(B_{\lambda^{k}}(x)\setminus B_{\lambda^{k+1}}(x)\big) \cap X \neq \emptyset \quad\mbox{or} \quad
	\big(B_{\lambda^{k+1}}(x)\setminus B_{\lambda^{k+2}}(x)\big) \cap X \neq \emptyset \quad \mbox{for every } x\in X.
	\end{equation}
	Let $x_0$ be an element of $X$ and, thanks to Lemma~\ref{lemma:cov-palla} let us cover $B_{\lambda^{k}}(x_0)$ with $(10 \lambda^{-2})^n$ balls of radius $\lambda^{k+2}/2$. Since $\lambda^{k}$ is greater than or equal to the diameter of $X$, we see that $X \subseteq B_{\lambda^{k}}(x_0)$; therefore, one of these balls, denoted by $\overline B$, contains more than $(10 \lambda^{-2})^{n(N-1)}$ points. Applying the inductive assumption on $\overline B \cap X$ after rescaling it by $\lambda^{k+2}$, we know that there exists a point $x\in X$ such that 
	$$ \H^0 \big(\big\{i\geq k+2: X \cap \big(B_{\lambda^{i}}(x) \setminus B_{\lambda^{i+1}}(x) \big) \neq \emptyset \big\} \big)>N-1.$$
	Since \eqref{eqn:alternativa} holds, we know that at least another annulus is occupied and we have proved the inductive statement.
\end{proof}

\begin{lemma}\label{lemma:cov-palla}
	For every $\mu \in (0,1)$ there exists a covering of $B_1 \subseteq \R^d$ made by $(5 \mu^{-1})^n$ balls of radius $\mu$.
\end{lemma}
\begin{proof}
 We apply the Vitali covering lemma: for every $x \in B_{1-\mu/5}$ we consider $B_\mu(x)$ and we extract a Vitali subcovering, namely a finite number of centers $x_i$, indexed by $i\in I$, such that the balls $B_{\mu/5}(x_i)$ are disjoint and the balls $B_{\mu}(x_i)$ cover $B_1$
The cardinality of $I$, that we denote with $N$, must satisfy
$N |B_{\mu/5}| \leq |B_1|$
since the balls $B_{\mu/5}(x_i)$ are disjoint and contained in $B_1$; this implies that $N \leq |B_1| |B_{\mu/5}|^{-1} = (5\mu^{-1})^n$.
\end{proof}
\begin{remark}
	The exponential dependence on $N$ in the estimate \eqref{ts:cover} of Lemma~\ref{lemma:covering} is optimal. Indeed, we can perform the following construction in $\R$. Let $N$ be a natural number and $\lambda = 1/4$. For every $e_1, ..., e_{N} \in \mathbb Z /  2\mathbb Z$, consider the point $p(e_1, ..., e_{N})= \sum_{i=1}^{N} (-1)^{e_i}2^{-12i} \in \R$.
	The collection consists of $2^{N}$ distinct points. We show that each point has at most $N$ dyadic annulai occupied by other points, by proving the following claim. 
	Let $\overline e_1 , ... ,\overline e_{N} \in \mathbb Z /  2\mathbb Z$ and let $j \in \{1,..., N\}$; then the points of the form
	$$ p(\overline e_1 , ... ,\overline e_{j-1} , \overline e_{j} +1, e_{j+1}... , e_{N} )$$ as $e_{j+1},..., e_{N}$ vary belong to $B_{2^{-12j+2}} (p({\overline e_1 , ... ,\overline e_{N}}))\setminus B_{2^{-12j}} (p({\overline e_1 , ... ,\overline e_{N}})) $, so that they belong to exactly one dyadic ring.
	
	Indeed, we have that
	\begin{equation}
	\begin{split}
	| p({\overline e_1 , ... , \overline e_{j-1} , \overline e_{j} +1, e_{j+1}... , e_{N} }) 
	&-
	p({\overline e_1 , ... ,\overline e_{N}}) -2 (-1)^{\overline{e}_j}2^{-12j})|
	\\&= 
	\Big |
	\sum_{k=j+1}^{N} (-1)^{\overline e_{k}}  2^{-12k} - 
	\sum_{k=j+1}^{N} (-1)^{ e_{k}}  2^{-12k}
	\Big|
	\\&
	\leq 2 \sum_{k=j+1}^{N}  2^{-12k}  \leq  2 \cdot 2^{-12 j} \sum_{k=1}^{\infty} 2^{-12k}\leq \frac 1 4 \cdot 2^{-12 j} 
	\end{split}
	\end{equation}
	and consequently
	$$2^{-12j} \leq \frac 7 4 \cdot 2^{-12 j}	\leq | p({\overline e_1 , ... , \overline e_{j-1} , \overline e_{j} +1,... , e_{N} }) 
	-	p({\overline e_1 , ... ,\overline e_{N}})|
	\leq\frac 9 4 \cdot 2^{-12 j}	\leq  2^{-12j+2}.
	$$
	
\end{remark}

\begin{proposition}\label{prop:quant-discreteness}
	Let $n,N\in \N$, $0< r_0, \Lambda$, $\Theta_0 \in (1,\infty)$ $\E$ be a $(\Lambda, r_0)$-minimizing cluster. Then there exist $\delta, \lambda \in [0,1/8]$, $r_1 \leq r_0$  (depending only on $n, N, \Lambda$ and $r_0$) such that if $x\in \Sigma^{0, \Theta_0}({\partial \E})$, $r \leq r_1$ and
$$M_\Lambda(\E, x,r) - M_\Lambda(\E, x, 4\lambda^2 r) \leq \delta,$$
then 
$$\Sigma^{0, \Theta_0}(\partial \E) \cap (B_{r/2}\setminus B_{\lambda r}) = \emptyset.$$
\end{proposition}
The previous proposition quantifies the fact that the singular set  $ \Sigma^{0, \Theta_0}({\partial \E})$ is discrete: indeed, as it is shown in the Proof of Theorem~\ref{thm:main}, by the monotonicity formula and by Proposition~\ref{prop:quant-discreteness} for every $x\in \Sigma^{0, \Theta_0}({\partial \E})$
the number of annulai of the form $B_{r_1(4\lambda)^n} (x) \setminus  B_{r_1(4\lambda)^{n-1}}(x)$ which intersect   $ \Sigma^{0, \Theta_0}({\partial \E})$ is finite.
\begin{proof}
%Let us call $\Theta_0$ the density of a cone in $\Sigma^{0, \Theta_0}_{\partial \E}$ and a
Assume by contradiction that there exist sequences $r_k \to 0$, $\lambda_k \to 0$, and a sequence of $(\Lambda, r_0)$-minimizing clusters $\E_k$ such that $0\in \Sigma^{0, \Theta_0}({\partial \E_k})$
\begin{equation}
	\label{eqn:cone}
	\lim_{k\to \infty} \big( M_\Lambda( \E_k, 0,r_k) - M_\Lambda(\E_k, 0, 4\lambda^2_k r_k) \big)= 0,
\end{equation}
and for every $k$ there exists a singular point $x_k \in \Sigma^{0,\Theta_0}({\partial \E_k}) \cap (B_{r_k/2}\setminus B_{\lambda_k r_k})$.
We consider the rescaled clusters $\E_k/|x_k|$ and, up to a subsequence, we assume that $\lim_{k\to \infty} x_k / |x_k| = x \in \partial B_1$ and that $\E_k/|x_k|$ converges to a generalized cluster $\E_{\infty}$, minimizing in the whole space. 
% since thanks to the monotonicity formula of Theorem~\ref{thm:monot} 
Moreover we have that the perimeter of the rescaled clusters in any ball is bounded independently on $k$
$$P\big((\E_k / |x_k|) \cap B_r\big)\leq N \omega_{n-2}r^{n-1}+ (n-1)\omega_{n-1}r^{n-1} +\Lambda r^n \qquad \mbox{for every $r>0$.}$$ Indeed, we can build a competitor for $\E_k / |x_k|$ in $B_r$ obtained by putting $(N-1)$ vertical discs splitting $B_r$ in $N$ parts such that the $i$-th part has the same volume as the $i$-th chamber of $\E$ in $B_r$. 

 Moreover, again by minimality of each $\E_k$ we have that no perimeter is lost when taking the limit as $k \to \infty$ in some fixed $B_r$: in other words, for every $r>0$ we have that 
$$ \lim_{k \to \infty}M_{\Lambda|x_k|}( \E_k/|x_k|, 0,r) =  M_0( \E_\infty, 0,r).$$ By the upper semicontinuity of the density,  we know that the point $x$ must have at least density $\Theta_0$; since $\Theta_0$ is the maximal density of an area-minimizing cone-like cluster, the density is exactly $\Theta_0$. In particular, this limit point still belongs to the $0$-stratum for the limit cluster $\E_{\infty}$. %of a triple junction in dimension 2, or a tetrahedral singularity in dimension 3. 

Since $\lambda_k r_k \leq |x_k| \leq r_k/2$ for every $k\in \N$, we have
\begin{equation}
\label{eqn:dicot1}
\lim_{k \to \infty} \frac{\lambda_k^2 r_k}{|x_k|} = 0, \qquad \frac{ r_k}{|x_k|} \geq 2.
\end{equation}
%or otherwise
%$$
%\liminf_{n \to \infty} \frac{|x_k|}{\lambda_k r_k} < \infty,
%$$
%which implies up to a not relabelled subsequence that
%\begin{equation}
%\label{eqn:dicot2}
%\lim_{n \to \infty} \frac{|x_k|}{r_k} = 0.
%\end{equation}
Hence, by \eqref{eqn:cone} and by the monotonicity of $M_{\Lambda |x_k|}( \frac{\E_k}{|x_k|}, 0, \cdot)$ % and the upper semicontinuity of the density,
 we know that
\begin{equation}
\begin{split}
0&= \lim_{k\to \infty} M_\Lambda( \E_k, 0,r_k) - M_\Lambda(\E_k, 0, 4\lambda_k^2 r_k) 
\\& =  \lim_{k\to \infty} \Big[ M_{\Lambda |x_k|}\Big( \frac{\E_k}{|x_k|}, 0,\frac{r_k}{|x_k|}\Big) - M_{\Lambda |x_k|}\Big(\frac{\E_k}{|x_k|}, 0, \frac{4\lambda_k^2 r_k}{|x_k|}\Big) \Big]
\\& \geq  \lim_{k\to \infty} \Big[ M_{\Lambda |x_k|}\Big( \frac{\E_k}{|x_k|}, 0,2\Big) - M_{\Lambda |x_k|}\Big(\frac{\E_k}{|x_k|}, 0, \frac{4\lambda_k^2 r_k}{|x_k|}\Big) \Big]
\\& =   M_0( \E_\infty, 0, 2) - M_0(\E_\infty, 0, 0)
\end{split}
\end{equation}
(we denote with $M_0(\E_\infty, 0, 0)$ the density at $0$, namely $\inf_{r>0} M_0(\E_\infty, 0, r) \geq 2$
%, and with $M_0(\E_\infty, 0, \infty)$ the density at $\infty$, namely $\sup_{r>0} M_0(\E_\infty, 0, r)\in \R_+\cup \{\infty\}$
).
Therefore, the previous quantity is exactly $0$ and the cluster $ \E_\infty$ is a cone-like cluster in $B_1$, again by the monotonicity formula in Theorem~\ref{thm:monot}. Hence the whole segment $x[0,2]$ is a singular line for the limit cluster with at least the same density as the origin. Then, we blow up the cluster $\E_\infty$ at $0$, finding a new cone-like cluster $\E_\infty'$ which coincides with $\E_{\infty}$ in $B_2$ and with the whole line $x\R_+$ as a singular line with density greater or equal than the one at the origin. So the cone-like cluster $\E_\infty'$ splits as a product with a factor that is given by the line $x\R$ (by Lemma~\ref{l:cones}), and $0$ cannot be a singularity of type $\Sigma^{0}({\partial \E_\infty})$.
%
%When \eqref{eqn:dicot2} is satisfied, by \eqref{eqn:cone} we deduce that
%\begin{equation}
%\begin{split}
%0&= \lim_{n\to \infty} M_{\Lambda}( \E_k, 0,r_k) - M_{\Lambda}(\E_k, 0, \lambda_k r_k) 
%\\& =  \lim_{n\to \infty} M_{\Lambda |x_k|}\Big( \frac{\E_k}{|x_k|}, 0,\frac{r_k}{|x_k|}\Big) - M_{\Lambda |x_k|}\Big(\frac{\E_k}{|x_k|}, 0, \frac{\lambda_k r_k}{|x_k|}\Big)
%\\& \geq  \lim_{n\to \infty} M_{\Lambda |x_k|}\Big( \frac{\E_k}{|x_k|}, 0,\frac{r_k}{|x_k|}\Big) - M_{\Lambda |x_k|}\Big(\frac{\E_k}{|x_k|}, 0, \frac 12\Big)
%\\& =   M_0( \E_\infty, 0, \infty) - M_0\Big(\E_\infty, 0, \frac 12\Big).
%\end{split}
%\end{equation}
%Therefore $ \E_\infty$ is a cone in $\R^n \setminus B_1$. Moreover, since we know that $0$ belongs to $\Sigma^{0,\Theta_0}({\partial \E_{\infty}})$ and since $  M_0( \E_\infty, 0, \cdot) $ is increasing, we have that  $  M_0( \E_\infty, 0, \infty) $ is greater or equal than the density $\Theta_0$.
%  Blowing $\E_\infty$ down at the origin, we find another minimizing cone, with the density $  M_0( \E_\infty, 0, \infty) $. Since we know all the possible singularities at the origin, we know that $  M_0( \E_\infty, 0, \infty) $ must coincide with the density of the triple junction in dimension $2$, or of the tetrahedral point in dimension $3$. Therefore the blow-down coincides with this singular point as well; hence $M_0( \E_\infty, 0, \infty)$ is constant, and the cone coincides with its blow-up at the origin. This does not allow for singular points of type $\Sigma^{0}_{\partial \E}$ outside the origin.
\end{proof}

\begin{proof}[Proof of Theorem~\ref{thm:main}] By scaling the cluster $\E$ of a factor $r_0$ and by translation, we reduce to the case $r_0= 1 $, namely we want to prove that if $\E$ is a $(\Lambda, 1)$-minimizing cluster there exist $C:= C(n, \Lambda)$,  $r_1:=r_1(n, \Lambda)$ such that for every $R\in (0, r_1/2)$
	\begin{equation}\label{ts-rescaled}
	\H^{0} \big(\Sigma^{0, \Theta_0}(\partial \E)\cap B_R(0)\big) \leq C^{1+\frac{P(\E; B_{R}(0))}{R^{n-1}}}.
	\end{equation}

	Let $\delta, \lambda \in [0,1/8]$, $r_1 \leq 1$ be as in Proposition~\ref{prop:quant-discreteness} (depending only on $n$ and $\Lambda$) and let $R\leq r_1/2$.
	We apply the result with $r=2R (2\lambda)^n$; we deduce that if for some $n \in \{0,1,...\}$ and $x\in \Sigma^{0, \Theta_0}({\partial \E})$
	$$\Sigma^{0, \Theta_0}(\partial \E)\cap \big( B_{R(2\lambda)^n} (x)\setminus  B_{R(2\lambda)^{n+1}} (x) \big) \neq \emptyset, %\quad \Rightarrow% \mbox{then} \quad 
	$$
	then
	\begin{equation}
	\label{eqn:drop}	
M_{\Lambda}(\E, x,2R (2 \lambda)^n) - M_{\Lambda}(\E, x,  2R (2 \lambda)^{n+2}) \geq \delta.
	\end{equation}
	Let us call $N_x$ the number of integers $n=0,1,...$ such that the condition in \eqref{eqn:drop} is satisfied. By applying \eqref{eqn:drop} for every $n\geq 0$, we have that
	\begin{equation}
	\label{eqn:nx}
	%\begin{split}
	\delta N_x \leq
	\sum_{n =0} ^\infty M_\Lambda(\E, x,2R (2 \lambda)^n) - M_\Lambda(\E, x, 2 R  (2 \lambda)^{n+2}),
	%\\& = \sum_{ n\, \textit{even}} M(\E, x,2R (4 \lambda)^n) - M(\E, x,  2R \lambda (4 \lambda)^n) 
	%+
	%\sum_{ n\, \textit{odd}} M(\E, x,2R (4 \lambda)^n) - M(\E, x,  2R \lambda (4 \lambda)^n)
	%\end{split}
	\end{equation}
	so that in particular the set  $\Sigma^{0, \Theta_0}({\partial \E})$ is discrete. 
	Since  by Theorem~\ref{thm:monot} the quantity $M(\E, x, \cdot)$ is monotonically increasing, and since the intervals in the set $\{ [2R (2 \lambda)^{n+2}, 2R (2\lambda)^n) : n \mbox{ odd} \}$ are all disjoint, and the same holds for $\{ [2R (2 \lambda)^{n+2}, 2R (2\lambda)^n) : n \mbox{ even} \}$, we deduce that
	$$ \sum_{ n\, \textit{even}} M_\Lambda(\E, x,2R (4 \lambda)^n) - M_\Lambda(\E, x,  2R \lambda (4 \lambda)^n) \leq M_\Lambda(\E, x,R),$$
	$$ \sum_{ n\, \textit{odd}} M_\Lambda(\E, x,2R (4 \lambda)^n) - M_\Lambda(\E, x,  2R \lambda (4 \lambda)^n) \leq M_\Lambda(\E, x,R)$$
	and by \eqref{eqn:nx}
	$$ N_x \leq\frac{2 M_\Lambda(\E, x,R)}{\delta} \leq \frac{2e^\Lambda P(\E; B_{2R}(x_0))}{\delta R^{n-1}}.
	$$
	Therefore, we apply Lemma~\ref{lemma:covering} to the rescaled set of points $ \big(\Sigma^{0, \Theta_0}(\partial \E)\cap B_{R}(0) \big) /(2R)$ and with $N ={2e^\Lambda P(\E; B_{2R}(x_0))}/{\delta R^{n-1}}$ 
	to deduce that, for every $x_0 \in \R^n$, \eqref{ts-rescaled} holds with a constant $C$ depending only on $n $ and $\lambda$.
	
\end{proof}

\begin{proof}[Proof of Corollary ~\ref{cor:cluster23}]
	By Theorem~\ref{thm:noto-cluster}, we know that there exists a bound on the diameter of the isoperimetric cluster solving problem \eqref{isoperimetric problem} with volume constraint $m \in [m_0, M_0]$.
	and that there exist $\Lambda,r_0$ such that each volume-constrained minimizer $\E$ is also $(\Lambda,r_0)$-minimizing.
	By Theorem~\ref{thm:main} there exists $r_1>0$ such that in each ball of radius $r_1$ there is only a quantified number of triple junctions in dimension $n=2$ (resp. tetrahedral points in dimension $n=3$); covering the cluster with balls of radius $r_1$ and using the bound on the diameter, we obtain \eqref{ts:quant-cluster}.
	
	Let us now fix $n=2$ and let us work in the context of the structure Theorem~\ref{thm:noto2}. Since the number of triple junctions is finite, there exists only a finite number of equivalence classes when we see $\partial \E$ as a non-oriented graph and the classes are according to graph homeomorphisms (namely, bijective maps sending vertices to vertices and edges between two given vertices in the  edge between corresponding vertices). Each graph homeomorphism between the graphs corresponding to $\partial \E $ and $\partial \E'$ can be moreover extended to a homeomorphism between $\partial \E $ and $\partial \E'$ by parameterizing each edge by arc-length.
	
	When $n=3$ (see Theorem~\ref{thm:noto3}), we notice that the density of tetrahedral points is strictly bigger than the density at triple junctions. Indeed, the second one equals $3/2$, as for the corresponding singularity in $\R^2$, by coarea formula. The first one can be computed by noticing that the side of a tetrahedron inscribed in the unit circle is $2^{3/2} 3^{-1/2}$ and that the angle at the circumcenter in a triangle made with two vertices of the tetrahedron is $2 \arccos(2^{1/2} 3^{-1/2})$, so that the density of the tetrahedron is $12 \arccos(2^{1/2} 3^{-1/2}) \pi^{-1} \simeq 2,3096$.
In this context, we consider graph homeomorphisms which preserve not only the edges (as in the two dimensional case), but also the faces between a certain subset of vertices.
	Let us consider one of the finitely many equivalence classes (up to these graph homeomorphisms). 
	Finally, by the Jordan curve theorem applied to the continuous curve that bounds a face in $\partial \E$, we can build an homeomorphism between this face and the unit disk that agrees with the arc-length parametrization on the boundary. Doing it also to the corresponding face of the cluster $\E'$, we can compose the two homeomorphisms to obtain a continuous, invertible map between the corresponding faces. Repeating this procedure on each face, we obtain an homeomorphism between $\partial \E$ and $\partial \E'$. 	
\end{proof}

The following proposition exploits a simplified version of Proposition~\ref{prop:quant-discreteness} in the context of area-minimizing hypersurfaces in $\R^8$.
\begin{proposition}\label{prop:quant-discreteness-8}
	Let $E$ be a set locally minimizing the perimeter functional in $\R^8$. 
	Then there exist universal constants $\delta, \lambda \in [0,1/8]$ such that if $x\in {\rm Sing} E$, $r>0$ and
	$$\frac{P(E; B_r)}{r^{n-1}} - \frac{P(E; B_{4\lambda^2 r})}{(4\lambda^2 r)^{n-1}} \leq \delta,$$
	then 
	$${\rm Sing}(E) \cap (B_{r/2}\setminus B_{\lambda r}) = \emptyset.$$
\end{proposition}
\begin{proof} By scaling, we can assume $r= 1$.
	Assume by contradiction that there exists a sequence $\lambda_k \to 0$, and a sequence $\{E_k\}_{k\in \N}$  of sets locally minimizing the perimeter functional such that $0\in {\rm Sing} (E_k)$ and such that
	\begin{equation*}
	\lim_{k\to \infty} \Big( {P(E_k; B_1)} - \frac{P(E_k; B_{4\lambda_k^2})}{(4\lambda_k^2)^{n-1}} \Big) = 0,
	\end{equation*}
	
	and a sequence of singular points $x_k \in {\rm Sing} (E_k) \cap (B_{1/2}\setminus B_{\lambda_k })$.
	We consider the rescaled sets $E_k/|x_k|$ and, up to a subsequence, we have that $\lim_{k\to \infty} x_k / |x_k| = x \in \partial B_1$ and that $E_k/|x_k|$ converges to a limit cone $E_{\infty}$ in $L^1(B_1)$, minimizing in $B_1$. 
	
	By the upper semicontinuity of the density and by the fact that singular points have density greater than $1+\eps$ for some $\eps>0$, both the origin and the point $x$ are singular in the limit cone. Therefore, the whole segment $x [0,1]$ is singular but this gives a contradiction since the singular set of any $7$-dimensional hypersurface in $\R^8$ is discrete.
	
\end{proof}
\begin{proof}[Proof of Theorem~\ref{thm:8}]
	The proof follows from Proposition~\ref{prop:quant-discreteness-8} and Lemma~\ref{lemma:covering} as in the proof of Theorem~\ref{thm:main}.
\end{proof}

%\cite{Almgren68,taylor76,CiLeMaIC1,CiLeMaIC2,nava1,nava2,ghsp}

\bibliography{references}
\bibliographystyle{is-alpha}

%\bibitem{lam} {\sc F. J. Almgren:} {\em Existence and regularity almost everywhere of solutions to elliptic variational problems with constraints}, Mem. Amer. Math. Soc., {\bf 4} (1976).

%\bibitem{tay} {\sc J. E. Taylor:} {\em The structure of singularities in soap-bubble-like and soap-film-like minimal surfaces},
%Ann. of Math., {\bf 103} (1976) 489--539.

\end{document}